\documentclass[reqno,11pt]{amsart}

\newtheorem{theorem}{Theorem}[section]
\newtheorem{lemma}[theorem]{Lemma}
\newtheorem{proposition}[theorem]{Proposition}

\theoremstyle{definition}
\newtheorem{definition}[theorem]{Definition}
\newtheorem{remark}[theorem]{Remark}

\numberwithin{equation}{section}
\usepackage{color}
\definecolor{MyLinkColor}{rgb}{0,0,0.4}
\usepackage[colorlinks,linkcolor=MyLinkColor,citecolor=MyLinkColor, urlcolor=black]{hyperref}
\usepackage{mathtools}
\usepackage{empheq}
\usepackage{mathrsfs}

\def\R{{\mathbb R}}
\def\0{\Omega}
\def\dx{\,{\rm d}x}
\def\dy{\,{\rm d}y}
\def\ds{\,{\rm d}s}

\def\dphi{\,{\rm d}\varphi}

\def\dtheta{\,{\rm d}\theta}
\def\dsi{\,{\rm d}\sigma}

\def\C{\mathcal{C}}

\def\T{\mathcal{T}}

\def\p{\partial}

\def\E{\mathcal{E}}
\def\ee{{\bf e}}
\def\n{{\bf n}}

\def\S{\mathbb{S}}
\def\x{{\bf x}}
\def\U{{\bf U}}

\DeclareMathOperator{\e}{e}

\makeatletter
\@namedef{subjclassname@2020}{\textup{2020} Mathematics Subject Classification}
\makeatother

\calclayout

\begin{document}
\title[Existence, uniqueness and stability for a model of the ACC]
{Global-in-time existence, uniqueness and stability of solutions to a model of the Antarctic Circumpolar Current}

\author[L. Roberti]{Luigi Roberti}
\address{Faculty of Mathematics, University of Vienna, Oskar--Morgenstern--Platz 1, 1090 Vienna, Austria}
\email{\href{mailto:luigi.roberti@univie.ac.at}{luigi.roberti@univie.ac.at}}

\author[E. Stefanescu]{Eduard Stefanescu}
\address{Institut f\"ur Analysis und Zahlentheorie, TU Graz, Steyrergasse 30, 8010 Graz, Austria}
\email{\href{mailto:eduard.stefanescu@tugraz.at}{eduard.stefanescu@tugraz.at}}


\begin{abstract}
We consider a model for the Antarctic Circumpolar Current in rotating spherical coordinates. After establishing global-in-time existence and uniqueness of classical solutions, we turn our attention to the issue of stability of a class of steady zonal solutions (i.e., time-independent solutions that vary only with latitude). By identifying suitable conserved quantities and combining them to construct a Lyapunov function, we prove a stability result.
\end{abstract}

\maketitle

\section{Introduction}

The Antarctic Circumpolar Current (ACC) is one of the Earth's most important currents. It transports approximately 140 million cubic meters of water per second over a distance of roughly 24,000\,km. Furthermore, it is the only current which flows all around the globe, unhindered by land masses, and connects the Atlantic, Pacific and Indian basins, thereby playing a crucial role in the global climate dynamics. It is located roughly between 40$^\circ$S and $60^\circ$S and has a typical width of about 2000\,km, its narrowest point (approximately 700\,km) being Drake Passage at the southern tip of South America; see, for instance, \cite{Gil82,Val17}.

The ACC has very peculiar features. Firstly, its meridional (North-South) velocity component is observed to be significantly smaller than the azimuthal (East-West) component, while the vertical velocity is altogether negligible. Moreover, unlike most other currents, it is not a single flow but instead comprises several localised, vertically coherent high-speed jets (typically 40-50 km wide, with speeds in excess of 1 m s$^{-1}$), reaching to the seafloor, separated by regions of low-speed flow, with average speeds below 20 cm s$^{-1}$ \cite{ConJoh16,KliNow01}. In short, the ACC is overall a deep-reaching, jet-structured flow that is mostly propagating zonally to the East around Antarctica, except where it is occasionally deviated by the underlying bottom topography, such as the Scotia Ridge (between 55$^\circ$W and 40$^\circ$W, downstream of South America), the Southwest Indian Ridge (between 20$^\circ$E and 30$^\circ$E) and the Pacific-Antarctic Ridge (between 145$^\circ$W and 120$^\circ$W)---see \cite{ConAbr23}. It is also worth noting that density stratification is much weaker in the Southern Ocean compared to equatorial region; this contributes to the vertical coherence of the ACC flow \cite{TomGod03}.

It is customary in oceanography to avoid using spherical coordinates by means of so-called tangent plane approximations. Most prominent amongst these are the $f$-plane approximation, which is derived from the Euler equations (or the Navier--Stokes equations, if viscosity effects are being considered) in the limit $L/R \to 0$, where $L$ is a typical length scale and $R$ is the Earth's radius, and the $\beta$-plane approximation, which seeks to improve the $f$-plane approximation by accounting for the meridional variation of the angular velocity with latitude, so long as the motion in the meridional direction is of comparatively small extension (see, for instance, \cite{Val17}). However, tangent plane approximations heavily rely on the flow being only of ``local" size; furthermore, it has beeen argued that the $\beta$-plane, albeit an improvement over the $f$-plane, doues not yield a consistent approximation in non-equatorial regions (cf. the discussion in \cite{Del11}). It is therefore of the utmost importance to work in spherical coordinates when dealing with flows of global extension such as the ACC. Let us remark that the Earth is in fact an oblate spheroid, but the small difference between the equatorial (6378\,km) and polar (6356\,km) radii does not significantly affect the dynamics (see, e.g., \cite{ConJoh21,Wun15}), so that we may approximate the Earth's surface by a perfect sphere.

The starting point are the Euler equations in spherical coordinates, in a frame of reference that is rotating with the Earth. The full equations being qualitatively intractable, it is necessary to simplify them through suitable approximations. The model we are going to consider is derived by means of the \emph{thin-shell approximation}, which exploits the ocean's small aspect ratio (i.e., the ratio between the ocean's depth and the Earth's radius) to perform an asymptotic expansion in the full model; the leading order equations are then investigated. Although this somewhat simplifies the problem, this approximation makes no further assumptions beyond the small aspect ratio, and thus crucially retains the sphericity of the Earth's surface, unlike the aforementioned tangent plane approximations. The problem at hand is in particular fully nonlinear. This method was originally devised as a means to describing ocean gyres \cite{ConJoh17,ConKri19}, but has been since applied to other oceanic flows \cite{ConJoh23,Haz19,HazMar18}, including the Antarctic Circumpolar Current, as well as atmospheric flows (via a slightly different scaling---see \cite{ConGer22,ConJoh21,ConJoh22}). Nevertheless, as a consequence of the thin-shell approximation, the motion is essentially two-dimensional, as the vertical motion vanishes at leading order. This means that this model is only appropriate to describe the near surface flow; still, in spite of this, it is an useful tool to gain new insights and as such has been studied extensively in the past few years. In particular, several authors showed how the stereographic projection may be a useful tool by which to rewrite the problem in a way amenable to analysis---see, for instance, \cite{HazMar18,Mar19}, where this approach is formalised, as well as \cite{ChuMar21,Chu23}, where the reformulation is utilised to construct zonal solutions with certain properties.

It seems however that, to the best of the authors' knowledge, all results concerning this model have up to now only dealt with steady (i.e., time-independent) flows, with nearly exclusive focus on purely zonal solutions (with the only exception of the study \cite{Haz19}, which discusses existence, uniqueness and regularity of solutions to an elliptic partial differential equation which arises by applying the Mercator projection to equation \eqref{stationary_radial} below for particular choices of the vorticity). The present study aims at extending the analysis to the time dependent case, establishing well-posedness in a classical sense as well as long-time stability of certain steady zonal flows. The stereographic projection will be useful when tackling the issue of existence and uniqueness of classical solutions, but the stability properties are more conveniently dealt with in the original spherical coordinates. For recent developments concerning models that rely on approximations other than thin-shell, the reader is referred to \cite{ConAbr23,ConJoh16}, where the authors construct explicit steady azimuthal depth-dependent solutions.

It is also worth pointing out that the issues of stability of stationary solutions and existence of non-zonal stationary solutions for the Euler equations on a rotating sphere have been addressed in the papers \cite{ConGer22,Tay16}; however, these studies deal with flows on the whole sphere (with applications to stratospheric flows on planets, in the case of \cite{ConGer22}), whereas our work focuses on solutions defined on a proper subset of the sphere, which requires a drastically different approach due to the presence of boundary conditions.

\section{Preliminaries}

In the following, we will parametrise the unit sphere $\S^2$ (with the North Pole $N$ and the South Pole $S$ excised) via the spherical coordinates
 \begin{equation} \label{spherical_coordinates}
 [0,2\pi)\times\left(-\frac{\pi}{2},\frac{\pi}{2}\right) \to \S^2, \quad (\varphi,\theta) \mapsto (\cos\varphi\cos\theta,\,\sin\varphi\cos\theta,\,\sin\theta);
 \end{equation}
in ``geophysical" terms, $\varphi$ denotes the longitude, whereas $\theta$ is the latitude. These coordinates provide us with a smooth field of tangent vectors
 \[\mathbf{e}_{\varphi}=\frac{1}{\cos \theta} \partial_{\varphi}, \quad \mathbf{e}_\theta=\partial_\theta,\]
which are a basis of the tangent space $T_X\S^2$ at $X\in\S^2\setminus\{N,S\}$. In these coordinates the Riemannian volume element and the classical differential operators (gradient, divergence and Laplace--Beltrami operator) are given by 
 \begin{align*}
 &\dsi=\cos \theta\dphi\dtheta,\\
 & \operatorname{grad}\psi=\frac{\p_{\varphi} \psi}{\cos \theta} \,\mathbf{e}_{\varphi}+\partial_\theta \psi\, \mathbf{e}_\theta, \\
 & \operatorname{div}\left(F^{\varphi}\mathbf{e}_{\varphi}+F^\theta\mathbf{e}_\theta\right)=\frac{1}{\cos \theta}[\p_{\varphi} F^{\varphi}+\p_\theta(F^\theta\cos \theta)], \\
 & \Delta \psi = \p_\theta^2 \psi-\tan \theta \p_\theta \psi+\frac{1}{(\cos \theta)^2} \p_{\varphi}^2 \psi = \frac{1}{\cos\theta}\p_\theta(\cos\theta\p_\theta\psi) + \frac{1}{(\cos \theta)^2}\p_{\varphi}^2\psi.
 \end{align*}
 
Henceforth derivatives will be denoted by subscripts for brevity. As mentioned in the Introduction, the equations that we are going to consider in this paper have been derived as an asymptotic model from the Euler equations in spherical coordinates, in a frame of reference that is rotating with the Earth.  Since the derivation of the model can be found at several places in the literature (see, for instance, \cite{ConJoh17,ConKri19,Mar19}), we will omit it here. Instead, we will merely write down the equations for the leading-order dynamics: these are, essentially, the non-dimensional incompressible Euler equations on the surface of a rotating sphere, written in the spherical coordinates \eqref{spherical_coordinates}, given by
 \begin{subequations} \label{Euler_spherical}
 \begin{align}
 & u_t + \frac{uu_\varphi}{\cos\theta} + vu_\theta - uv\tan\theta - 2\omega v\sin\theta = -\frac{p_\varphi}{\cos\theta} \label{Euler_spherical_1} \\
 & v_t + \frac{uv_\varphi}{\cos\theta} + vv_\theta + u^2\tan\theta + 2\omega u\sin\theta = -p_\theta \\
 & u_\varphi + (v\cos\theta)_\theta = 0, \label{incompressible}
 \end{align}
 \end{subequations}
where $\omega>0$ is the (non-dimensional) rotation rate, $(u, v)$ are the velocity components along the unit vectors $(\mathbf{e}_\varphi, \mathbf{e}_{\theta})$, and $p$ is the pressure. For later reference, let us remark that $\omega = \Omega^{\prime}R^{\prime}/U^{\prime}$, where $\Omega^{\prime} \approx 7.29 \times 10^{-5}\, \mathrm{rad}\, \mathrm{s}^{-1}$ is the rate of rotation of the Earth, $R^{\prime}\approx6378$\,km is the mean radius of the Earth, and $U' = 10^{-1}\,{\rm m\, s}^{-1}$ is a typical velocity scale that is relevant when performing the scaling that eventually leads to \eqref{Euler_spherical}. With this choice of $U'$, we have $\omega \approx 4650$. Smooth solutions to \eqref{Euler_spherical} must necessarily be $2\pi$-periodic in $\varphi$, since we are interested in an annular region $\C$ on the sphere situated in the Southern hemisphere and delimited by fixed latitudes $-\frac{\pi}{2}<\theta_1<\theta_2<0$:
 \[\C = \left\{(\varphi,\theta) \in [0,2\pi)\times\left(\theta_1,\theta_2\right)\right\}.\]
As boundary conditions, we take
 \begin{equation} \label{BC_spherical}
 v(\varphi,\theta_1,t) = v(\varphi,\theta_2,t) = 0 \quad \text{for all } \varphi\in [0,2\pi) \text{ and all } t\geq 0.
 \end{equation}
The existence of classical solutions to this problem is established in § \ref{classical_solutions} and Appendix \ref{appendix_classical_solutions} below.

We will denote by $\0$ the vorticity associated to $(u,v)$, namely,
 \[\0 = \frac{1}{\cos\theta}[v_\varphi - (u\cos\theta)_\theta].\]
Since the domain $\C$ is doubly connected and the velocity field $u\,\mathbf{e}_\varphi + v\,\mathbf{e}_\theta$ is tangential to the boundary $\p\C = \{\theta=\theta_1\}\cup\{\theta=\theta_2\}$ of $\C$, equation \eqref{incompressible} implies at each time $t$ the existence of a stream function $\psi(\varphi,\theta,t)$ with the property that
 \begin{equation} \label{stream_function}
 u = -\psi_\theta \quad \text{and} \quad v = \frac{1}{\cos\theta}\psi_\varphi.
 \end{equation}
At each fixed $t$, the stream function $\psi$ is constant on each of the two circles that $\p\C$ is comprised of, the constant value on $\{\theta = \theta_1\}$ being different than the constant value on $\{\theta = \theta_2\}$ (cf. \cite{HamNad23}). Since, at each time $t$, the stream function is defined up to an additive constant (depending on $t$), we may determine $\psi$ uniquely by prescribing, for instance, a fixed value $\psi_2\in\R$ to be attained by $\psi$ on $\{\theta = \theta_2\}$, so that
 \begin{equation} \label{BC_polar}
 \psi(\varphi,\theta_1,t) = \psi_1(t) \quad \text{and } \quad \psi(\varphi,\theta_2,t) = \psi_2 \quad \text{for all } \varphi\in[0,2\pi) \text{ and all } t \geq 0,
 \end{equation}
with $\psi_1(t) \neq \psi_2$ for all $t$. In terms of $\psi$, the equations of motion \eqref{Euler_spherical} can be reformulated (by eliminating the pressure $p$; cf. \cite{ConGer22}) as
 \begin{equation} \label{main_time_dependent}
 (\Delta\psi)_t + \frac{1}{\cos\theta}\big[\psi_\varphi\p_\theta - \psi_\theta\p_\varphi\big](\Delta\psi + 2\omega\sin\theta) = 0 \quad \text{in } \C,
 \end{equation}
with boundary conditions \eqref{BC_polar}. Note that, by the rank theorem \cite{Abr88}, stationary (that is, time-independent) solutions to \eqref{main_time_dependent} are given by the solutions to the equation
 \begin{equation} \label{stationary_radial}
 \Delta\psi + 2\omega\sin\theta = F(\psi) \quad \text{in } \C,
 \end{equation}
as long as the gradient of $\psi$ does not vanish anywhere in the annulus, where $F$ is an arbitrary $C^1$ function on $\R$. Physically, the meaning of this equation is that the total vorticity $\0$ (which is readily checked to be equal to $\Delta\psi$) is given by the sum of two components: that due to the rotation of the Earth (namely, $-2\omega\sin\theta$), and that due to the motion of the ocean ($F(\psi)$, the so-called \emph{oceanic vorticity}). In later sections, we will consider affine oceanic vorticities of the form
 \begin{equation} \label{F(psi)}
 F(\psi) = -\lambda\psi + \Upsilon
 \end{equation}
for real parameters $\lambda,\Upsilon\in\R$. This choice is not merely for mathematical convenience, but is also physically motivated. In fact, the main sources of oceanic vorticity are wind stress and tidal currents, which are modelled by constant non-zero vorticity, whose sign determines whether the underlying shear currents correspond to ebb or flow tide; see, e.g., the discussion in \cite{TelPer88}. Since it can be of interest to investigate perturbations of constant vorticities as well, the first, most natural generalisation is to consider linear vorticities (which would be the next order in the Taylor expansion of a general vorticity function).

\section{Notation and main results}

It is convenient to fix some notation that will be used throughout the paper, before we formulate the results, which are going to be proven in the sequel. Let $\C$ be the subset of $\S^2$ defined above, and let $\overline{\C}$ denote the set
 \[\overline{\C} = \left\{(\varphi,\theta) \in [0,2\pi)\times\left[\theta_1,\theta_2\right]\right\}.\]
Let $f : \C\to\R$ and $\vartheta\in (0,1)$. We say that $f\in C^{0+\vartheta}(\overline{\C})$ if there exists $C>0$ such that
 \[|f(\tilde{\varphi}_1,\tilde{\theta}_1) - f(\tilde{\varphi}_2,\tilde{\theta}_2)| \leq C\,d((\tilde{\varphi}_1,\tilde{\theta}_1),(\tilde{\varphi}_2,\tilde{\theta}_2))^\vartheta \quad \text{for all } (\tilde{\varphi}_1,\tilde{\theta}_1),(\tilde{\varphi}_2,\tilde{\theta}_2) \in \overline{\C},\]
where $d(\cdot,\cdot)$ denotes the geodesic length on the sphere (i.e., $d((\tilde{\varphi}_1,\tilde{\theta}_1),(\tilde{\varphi}_2,\tilde{\theta}_2))$ is the length of the shortest arc on the sphere connecting the points with spherical coordinates $(\tilde{\varphi}_1,\tilde{\theta}_1),(\tilde{\varphi}_2,\tilde{\theta}_2)$). Given $k\in\mathbb{N}$, we say that $f\in C^{k+\vartheta}(\overline{\C})$ if all the derivatives of $f$ with respect of $(\varphi,\theta)$ up to order $k$ are in $C^{0+\vartheta}(\overline{\C})$. A vector-valued function ${\bf f} : \C\to\bigcup_{X\in\C}T_X\S^2$, ${\bf f}(\varphi,\theta) = f(\varphi,\theta)\ee_\varphi + g(\varphi,\theta)\ee_\theta$ is said to be in $C^{k+\vartheta}(\overline{\C})$ if both $f$ and $g$ are.

Let us now summarise the main results of this paper. First, we will deal with the issue of existence and uniqueness of classical solutions to our problem. 

\begin{theorem}[Global classical solutions] \label{Main_Theorem_1}
Suppose that ${\bf u}_0 = u_0\,\ee_\varphi + v_0\,\ee_\theta \in C^{1+\vartheta}(\overline{\C})$ for some $\vartheta\in(0,1)$, $\operatorname{div}{\bf u}_0 = 0$ in $\C$, and $v_0 = 0$ on $\p\C = \{\theta = \theta_1\}\cup\{\theta = \theta_2\}$. Furthermore, let $T > 0$. Then there exists a solution $(u\,\ee_\varphi + v\,\ee_\theta,p)$ to \eqref{Euler_spherical}--\eqref{BC_spherical}, with $(u\,\ee_\varphi + v\,\ee_\theta)|_{t=0} = {\bf u}_0$, such that $u$, $v$, $p$, and all their derivatives as appear in \eqref{Euler_spherical} belong to $C(\overline{\C}\times[0,T])$. Such a solution is unique up to an arbitrary function of $t$ which may be added to $p$.
\end{theorem}

The strategy, carried out in Section \ref{classical_solutions}, is to reformulate the problem \eqref{Euler_spherical}--\eqref{BC_spherical} more conveniently in Cartesian coordinates by means of the stereographic projection, which allows us to apply classical arguments. The Cartesian version of Theorem \ref{Main_Theorem_1} is formulated in Theorem \ref{global_well-posedness} below. The actual proof, which is standard and follows closely the classical paper \cite{Kat67}, is then postponed to Appendix \ref{appendix_classical_solutions}.

The central result of this paper concerns the stability of certain zonal steady flows, denoted by $\Psi_{\lambda,\Upsilon}$, which solve \eqref{BC_polar}--\eqref{stationary_radial} with $F$ of the form \eqref{F(psi)}. These flows are constructed in Section \ref{radial_Sturm-Liouville} using Sturm--Liouville theory. The stability result, proved in Section \ref{stability} by exploiting some conservation laws to construct a Lyapunov function, reads as follows.

\begin{theorem}[Stability of zonal flows] \label{Main_Theorem_2}
Given $\lambda \leq 0$ and $\Upsilon\in\R$, let $\Psi_{\lambda,\Upsilon}$ be the solution of \eqref{radial_problem} below (for some real constants $\psi_2\neq \psi_1$), whose existence and uniqueness is shown in Proposition \ref{one-parameter_family}; let $\mathbf{u}^*_{\lambda,\Upsilon} = u^*_{\lambda,\Upsilon}\mathbf{e}_\varphi$ denote the corresponding steady zonal solution to \eqref{Euler_spherical} and \eqref{BC_spherical} (via \eqref{stream_function}), with associated vorticity $\0^*_{\lambda,\Upsilon}$. Suppose that $\mathbf{u}(t) = u(t)\mathbf{e}_\varphi + v(t)\mathbf{e}_\theta$ is a (time-dependent) solution to \eqref{Euler_spherical} and \eqref{BC_spherical} with initial data $\mathbf{u}|_{t=0} = \mathbf{u}_0 = u_0\mathbf{e}_\varphi + v_0\mathbf{e}_\theta$, and let $\0(t)$ denote the associated vorticity at time $t$, with initial value $\0|_{t=0} = \0_0$. Then
 \begin{equation} \label{stability_conclusion}
 -\lambda\|\mathbf{u}(t) - \mathbf{u}^*_{\lambda,\Upsilon}\|^2_{L^2(\C)} + \|\0(t)-\0^*_{\lambda,\Upsilon}\|^2_{L^2(\C)} = -\lambda\|\mathbf{u}_0-\mathbf{u}^*_{\lambda,\Upsilon}\|^2_{L^2(\C)} + \|\0_0-\0^*_{\lambda,\Upsilon}\|^2_{L^2(\C)}
 \end{equation}
for all $t\geq 0$.
\end{theorem}

\section{A Cartesian reformulation and classical solutions} \label{classical_solutions}

The issue of existence of classical solutions to problem \eqref{Euler_spherical}--\eqref{BC_spherical} in the domain $\C$ can be discussed most conveniently by transferring the problem into Cartesian coordinates in the plane; this will enable us to rely on classical arguments for the well-posedness of the Euler equations in bounded two-dimensional domains.

Let $(x,y)$ denote the Cartesian coordinates in $\R^2$. The corresponding standard unit vectors are denoted $\ee_x = (1,0)$ and $\ee_y = (0,1)$. Given a domain $D\subseteq\R^2$, with sufficiently smooth boundary, and $T>0$, we will write $Q_T = D\times[0,T]$ and $\overline{Q}_T = \overline{D}\times[0,T]$. For $k\in\mathbb{N}\cup\{0\}$, the spaces $C^k(\overline{D})$ and $C^k(\overline{Q}_T)$ are the spaces of $k$-times continuously differentiable functions on $\overline{D}$ respectively $\overline{Q}_T$. We set $C(\overline{D}) = C^0(\overline{D})$. For $\delta\in (0,1)$, the space $C^{k+\delta}(\overline{D})$ is comprised of all functions whose first $k$ derivatives all exist and are H\"older continuous with exponent $\delta$ on $\overline{D}$.

For a scalar function $h = h(\x)$ (where $\x=(x,y)\in D$), we shall write
 \[\nabla h = h_x\ee_x + h_y\ee_y \quad \text{and} \quad \nabla^\perp h = h_y\ee_x - h_x\ee_y,\]
whereas for a vector function ${\bf f} = f\,\ee_x + g\,\ee_y$ we will denote
 \[\nabla\cdot{\bf f} = f_x+g_y \quad \text{and} \quad \nabla\times{\bf f} = g_x - f_y,\]
the former being the divergence and the latter being the (scalar) curl. It follows, for the Laplacian in Cartesian coordinates,
 \[\Delta h = h_{xx} + h_{yy} = \nabla\cdot\nabla h = - \nabla\times\nabla^\perp h.\]

We begin by projecting the problem onto a planar region on the equatorial plane by means of a stereographic projection with center of projection situated at the North Pole. We introduce the change of coordinates
 \[x = \frac{\cos\theta}{1-\sin\theta}\cos\varphi, \quad y = \frac{\cos\theta}{1-\sin\theta}\sin\varphi,\]
$(x,y)$ being the Cartesian coordinates in the equatorial plane. This defines a bijection from the domain $\C$ onto an annulus
 \[D = \{(x,y)\in\R^2 : r_1^2 < x^2+y^2 < r_2^2\},\]
where $r_i = \cos\theta_i/(1-\sin\theta_i)$, $i\in\{1,2\}$. The inverse of this coordinate transformation is easily computed by observing that
 \[\tan\varphi = \frac{y}{x} \quad \text{and} \quad \sin\theta = -\frac{1-x^2-y^2}{1+x^2+y^2};\]
in particular, we have
 \begin{align*}
 \p_\varphi &= -y\p_x + x\p_y, \\
 \p_\theta &= \frac{1}{\cos\theta}(x\p_x + y\p_y) = \frac{1+x^2+y^2}{2\sqrt{x^2+y^2}}(x\p_x + y\p_y).
 \end{align*}
While this coordinate change might seem cumbersome at first, it will in fact allow us to rewrite the equations \eqref{Euler_spherical} in a simple Cartesian form. To this end, we introduce the new unknown $\U = (U,V)$ and $P$ as
 \begin{subequations} \label{new_unknowns}
 \begin{align}
 U(x,y,t) &= (1-\sin\theta)[v(\varphi,\theta,t)\cos\varphi - u(\varphi,\theta,t)\sin\varphi] \nonumber \\
          &= \frac{2[xv(\varphi(x,y),\theta(x,y),t) - yu(\varphi(x,y),\theta(x,y),t)]}{(1+x^2+y^2)\sqrt{x^2+y^2}}, \\ \label{U_def}
 V(x,y,t) &= (1-\sin\theta)[u(\varphi,\theta,t)\cos\varphi + v(\varphi,\theta,t)\sin\varphi], \nonumber \\
          &= \frac{2[xu(\varphi(x,y),\theta(x,y),t) + yv(\varphi(x,y),\theta(x,y),t)]}{(1+x^2+y^2)\sqrt{x^2+y^2}}, \\ \label{V_def}
 P(x,y,t) &= p(\varphi(x,y),\theta(x,y),t).
 \end{align}
 \end{subequations}
The motivation for the definition of $U$ and $V$ is as follows. Denoting once again by ${\bf e}_x,{\bf e}_y$ the standard Cartesian unit vectors and by ${\bf e}_r,{\bf e}_\varphi$ the standard polar unit vectors in the plane, we have, at each point $(\varphi,r)$,
 \begin{equation} \label{derivatives}
 \begin{aligned}
 \mathbf{e}_r &= \cos\varphi\,\mathbf{e}_x + \sin\varphi\,\mathbf{e}_y, \\
 \mathbf{e}_\varphi &= -\sin\varphi\,\mathbf{e}_x + \cos\varphi\,\mathbf{e}_y;
 \end{aligned}
 \end{equation}
in particular, for any vector ${\bf v} = v^\varphi{\bf e}_\varphi + v^r{\bf e}_r = v^x{\bf e}_x + v^y{\bf e}_y$, its polar components $(v^\varphi,v^r)$ and Cartesian components $(v^x,v^y)$ are related via
 \begin{align*}
 v^x &= v^r\cos\varphi - v^\varphi\sin\varphi, \\
 v^y &= v^\varphi\cos\varphi + v^r\sin\varphi.
 \end{align*}
The additional factor $(1-\sin\theta)$ in \eqref{U_def} and \eqref{V_def} is simply a scaling that turns out to be convenient to write the equations more compactly.

Let us introduce for brevity the functions
 \[\alpha(x,y) = \frac{(1+x^2+y^2)^2}{4} \quad \text{and} \quad \beta(x,y) = 2\omega\frac{1-x^2-y^2}{1+x^2+y^2}\]
(note that $\alpha,1/\alpha,\beta \in C^\infty(\overline{D})$) and the linear map $J:\R^2\to\R^2$ for which $J\ee_x = \ee_y$ and $J\ee_y = -\ee_x$. Plugging \eqref{derivatives} into \eqref{Euler_spherical}, rearranging terms, and using \eqref{new_unknowns}, a rather lengthy calculation eventually yields the 2D Euler-like equation
 \begin{subequations} \label{Euler_compact}
 \begin{align}
 & \U_t + (\alpha\U\cdot\nabla)\U + \frac{|\U|^2}{2}\nabla\alpha + \beta J\U = -\nabla P, \label{Euler_compact_1} \\
 & \nabla\cdot\U = 0 \label{Euler_compact_2}
 \end{align}
 \end{subequations}
in $D$, with boundary condition
 \begin{equation} \label{Euler_compact_3}
 \U\cdot\mathbf{n} = 0 \quad \text{on } \p D;
 \end{equation}
here, ${\bf n}$ denotes the outward unit normal vector on  the boundary $\p D$ of $D$. The accordingly transformed initial data will be denoted by
 \begin{equation} \label{initial_condition}
 \U|_{t=0} = \U_0,
 \end{equation}
for some sufficiently regular function $\U_0:D\to\R^2$. We then have the following result concerning the existence and uniqueness of classical solutions:

\begin{theorem} \label{global_well-posedness}
Suppose that $\U_0\in C^{1+\vartheta}(\overline{D})$ for some $\vartheta\in(0,1)$, $\nabla\cdot\U_0 = 0$, and $\U_0\cdot{\bf n} = 0$. Furthermore, let $T > 0$. Then there exists a solution $(\U,P)$ to \eqref{Euler_compact}--\eqref{initial_condition} such that $\U$, $P$, and all their derivatives as appear in \eqref{Euler_compact} belong to $C(\overline{Q}_T)$. Such a solution is unique up to an arbitrary function of $t$ which may be added to $P$.
\end{theorem}

The arguments required to prove this theorem are very similar in spirit to those in the classical paper \cite{Kat67}. Thus, as the proof of Theorem \ref{global_well-posedness} requires no new insights, it is postponed to Appendix \ref{appendix_classical_solutions}. Theorem \ref{Main_Theorem_1} is now an immediate consequence, in view of some simple observations. In fact, denoting by $(x,y)$ the Cartesian coordinates in the projection plane corresponding to the spherical coordinates $(\theta,\varphi)$, we have
 \begin{align*}
 U_0(x,y) = \delta(x,y)[xv_0(\theta,\varphi) - yu_0(\theta,\varphi)], \\
 V_0(x,y) = \delta(x,y)[xu_0(\theta,\varphi) + yv_0(\theta,\varphi)],
 \end{align*}
where $\delta(x,y) = 2[(1+x^2+y^2)\sqrt{x^2+y^2}]^{-1}$. The claim follows in view of the fact that $x\delta,y\delta\in C^\infty(\overline{D})$ (thus in particular $x\delta,y\delta\in C^{1+\vartheta}(\overline{D})$), and from the following lemma, since then ${\bf u}_0\in C^{1+\vartheta}(\overline{\C})$ implies $\U_0 \in  C^{1+\vartheta}(\overline{D})$.

\begin{lemma}
There exists a constant $C>0$ (depending only on $\C$, or equivalently $D$) such that
 \[d((\tilde{\varphi}_1,\tilde{\theta}_1),(\tilde{\varphi}_2,\tilde{\theta}_2)) \leq C\|(x_1,y_1)-(x_2,y_2)\|_2\]
for all  $(\tilde{\varphi}_1,\tilde{\theta}_1),(\tilde{\varphi}_2,\tilde{\theta}_2)\in\C$, with corresponding projections $(x_1,y_1),(x_2,y_2) \in D$ (where $\|\cdot\|_2$ denotes the Euclidean norm).
\end{lemma}
\begin{proof}
Let $(X_1,Y_1,Z_1),(X_2,Y_2,Z_2)\in\R^3$ be the Cartesian coordinates of points of $\S^2$ corresponding to $(\tilde{\varphi}_1,\tilde{\theta}_1),(\tilde{\varphi}_2,\tilde{\theta}_2)\in\C$ (via \eqref{spherical_coordinates}). Then
 \[(X_k,Y_k,Z_k) = \left(\frac{2x_k}{1+x_k^2+y_k^2},\frac{2y_k}{1+x_k^2+y_k^2},\frac{-1+x_k^2+y_k^2}{1+x_k^2+y_k^2}\right), \quad k\in\{1,2\}.\]
Some straightforward calculations yield
 \[\|(X_1,Y_1,Z_1) - (X_2,Y_2,Z_2)\|_1 \leq \frac{4(1+2r_2^2)}{(1+r_1^2)^2}\|(x_1,y_1) - (x_2,y_2)\|_1,\]
where $\|\cdot\|_1$ denotes the $\ell^1$-norm in $\R^n$, $n\in\{2,3\}$. The equivalence of norms on $\R^n$ now implies
 \[\|(X_1,Y_1,Z_1) - (X_2,Y_2,Z_2)\|_2 \leq \tilde{C}\|(x_1,y_1) - (x_2,y_2)\|_2\]
for some $\tilde{C} > 0$. Since
 \[\|(X_1,Y_1,Z_1) - (X_2,Y_2,Z_2)\|_2 \leq d((\tilde{\varphi}_1,\tilde{\theta}_1),(\tilde{\varphi}_2,\tilde{\theta}_2)) \leq \frac{\pi}{2}\|(X_1,Y_1,Z_1) - (X_2,Y_2,Z_2)\|_2\]
for all $(\tilde{\varphi}_1,\tilde{\theta}_1),(\tilde{\varphi}_2,\tilde{\theta}_2)\in\C$, the claim follows with $C = \frac{\pi}{2}\tilde{C}$.
\end{proof}

\section{Existence and stability of a class of zonal solutions}

In this section we will seek zonal solution of equation \eqref{stationary_radial} (i.e., solutions that depend solely on $\theta$) in the case where $F$ is given by \eqref{F(psi)}. In other words, for given $\lambda,\Upsilon \in\R$, we are seeking a function $\Psi_{\lambda,\Upsilon} = \Psi_{\lambda,\Upsilon}(\theta)$ which solves the following ODE problem (where a prime denotes a derivative with respect to $\theta$):
 \begin{align}\label{radial_problem}
 \begin{split}
 &\big(\Psi_{\lambda,\Upsilon}'(\theta)\cos\theta\big)' = -\lambda\Psi_{\lambda,\Upsilon}(\theta)\cos\theta + \Upsilon\cos\theta - \omega\sin2\theta, \quad \theta \in (\theta_1,\theta_2), \\
 &\Psi_{\lambda,\Upsilon}(\theta_1) = \psi_1, \ \Psi_{\lambda,\Upsilon}(\theta_2) = \psi_2,
 \end{split}
 \end{align}
for $\psi_2\neq \psi_1$.

\subsection{Zonal solutions via Sturm--Liouville} \label{radial_Sturm-Liouville}

Let us recollect some well-known facts about Sturm--Liouville problems (see, e.g., \cite{Tes12} for proofs):

\begin{lemma} \label{Sturm-Liouville_lemma}
Let $a,b\in\R$, $a<b$, and suppose that the functions $p\in C^1([a,b])$ and $w\in C([a,b])$ are both positive: $p(x)> 0$ and $w(x)>0$ for all $x\in[a,b]$. Also, let $q\in C([a,b])$ and $\alpha,\beta,\gamma,\delta\in\R$ with $(\alpha,\beta) \neq (0,0) \neq (\gamma,\delta)$. Then there exists a sequence $\mu_1<\mu_2<\ldots\to+\infty$ and, for each $n\in\mathbb{N}$, a function $y_n\in C^2([a,b])$ such that
 \begin{align*}
 & (p(x)y'_n(x))' + q(x)y_n(x) = - \mu_n w(x)y_n(x), \quad x\in (a,b), \\
 & \alpha y_n(a) + \beta y_n'(a) = 0, \\
 & \gamma y_n(b) + \delta y_n'(b) = 0.
 \end{align*}
In addition, all solutions to this problem are a multiple of $y_n$ (i.e., each eigenvalue $\mu_n$ is simple), and $y_n$ has exactly $n-1$ zeros. Moreover, the set $\{y_n : n\in\mathbb{N}\}$ is an orthonormal basis of the weighted $L^2$-space $L^2([a,b];w(x)\dx)$ with respect to the inner product
 \[\langle f,g\rangle = \int_a^b f(x)g(x)w(x)\dx\]
(that is, $\langle y_n,y_m\rangle = \delta_{mn}$). The eigenfunction expansion
 \[f = \sum_{n=1}^\infty\langle f,y_n\rangle y_n\]
converges uniformly for all $f\in C^2([a,b])$ with $f(a) = f(b) = 0$. Finally, each eigenvalue $\mu_n$ can be recovered from a corresponding eigenfunction $y_n$ via the \emph{Rayleigh quotient}:
 \begin{equation} \label{Rayleigh}
 \mu_n = \frac{-py_ny_n'|^a_b + \int_a^b\big[p(y_n')^2 - qy_n^2\big]\dx}{\langle y_n,y_n\rangle}.
 \end{equation}
\end{lemma}

At several places throughout the paper (in later sections as well) we are going to additionally require some knowledge about inhomogeneous Sturm--Liouville problems; this is the object of the next lemma.

\begin{lemma} \label{lemma_inhomogeneous}
Let $a,b\in\R$, $a<b$, and suppose that the functions $p\in C^1([a,b])$ and $w\in C([a,b])$ are both positive: $p(x)> 0$ and $w(x)>0$ for all $x\in[a,b]$. Also, let $q,h\in C([a,b])$, and $\alpha,\beta,\gamma,\delta\in\R$ with $(\alpha,\beta) \neq (0,0) \neq (\gamma,\delta)$. For $\mu\in\R$, consider the inhomogeneous Sturm--Liouville problem
 \begin{subequations} \label{inhomogeneous}
 \begin{align}
 & (p(x)y'(x))' + q(x)y(x) = - \mu w(x)y(x) + h(x), \quad x\in (a,b), \label{inhomogeneous_a} \\
 & \alpha y(a) + \beta y'(a) = 0, \\
 & \gamma y(b) + \delta y'(b) = 0.
 \end{align}
 \end{subequations}
Let $\mu_1 < \mu_2 < \ldots$ be the eigenvalues of the homogeneous problem ($h\equiv 0$), with corresponding eigenfunctions $y_n$ ($n\in\mathbb{N}$), scaled so that $\langle y_n,y_n\rangle = 1$ for all $n\in\mathbb{N}$. Then the inhomogeneous problem \eqref{inhomogeneous} has a unique solution if and only if $\mu \notin \{\mu_n : n\in\mathbb{N}\}$. If $\mu = \mu_k$ for some $k\in\mathbb{N}$, then:
 \begin{itemize}
 \item[(a)] either $\langle\frac{h}{w},y_k\rangle \neq 0$, in which case there is no solution, or
 \item[(b)] $\langle\frac{h}{w},y_k\rangle = 0$, in which case there are infinitely many solutions, which are of the form
  \[y(x) = \sum_{n= 1}^\infty b_n y_n(x),\]
 where
  \[b_n = \frac{\langle\frac{h}{w},y_k\rangle}{\mu - \mu_n} \quad \text{if } n \neq k\]
 and $b_k$ is an arbitrary real number.
 \end{itemize}
\end{lemma}
\begin{proof}
Suppose that $\mu\notin\{\mu_n : n\in\mathbb{N}\}$. Let $c_n = \langle\frac{h}{w},y_n\rangle$, thus
 \[\frac{h(x)}{w(x)} = \sum_{n=1}^\infty c_n y_n(x)\]
(with convergence in the Hilbert space $L^2([a,b];w(x)\dx)$). Any solution $y$ of the inhomogeneous problem \eqref{inhomogeneous} must have an eigenfunction expansion
 \[y = \sum_{n=1}^\infty b_n y_n;\]
we now show that appropriate coefficients $(b_n)_{n\in\mathbb{N}}$ can indeed be uniquely determined. Since
 \[(p(x)y_n'(x))' + q(x)y_n(x) + \mu w(x)y_n(x) = (\mu - \mu_n)w(x)y_n(x),\]
plugging into \eqref{inhomogeneous_a} yields
 \[\frac{h(x)}{w(x)} = \sum_{n=1}^\infty b_n(\mu-\mu_n)y_n(x),\]
hence
 \begin{equation} \label{coefficients}
 (\mu-\mu_n)b_n = c_n, \quad n\in\mathbb{N},
 \end{equation}
and thus $b_n = \frac{c_n}{\mu-\mu_n}$ for all $n\in\mathbb{N}$.

If, on the other hand, $\mu = \mu_k$ for some $k\in\mathbb{N}$, and $c_k \neq 0$, then the equality \eqref{coefficients} cannot be satisfied for $n=k$, hence there is no solution in this case. Finally, if $\mu = \mu_k$ for some $k\in\mathbb{N}$ and $c_k = 0$, the equality \eqref{coefficients} is satisfied for any choice of the coefficient $b_k$; this concludes the proof.
\end{proof}

Combining the previous two lemmas, we obtain a family of zonal solutions to \eqref{radial_problem}.

\begin{proposition} \label{one-parameter_family}
Let $0 < \lambda_1 < \lambda_2 < \ldots\to+\infty$ be the sequence of eigenvalues (provided by Lemma \ref{Sturm-Liouville_lemma}) of the problem
 \begin{align} \label{Sturm-Liouville_homogeneous}
 \begin{split}
 &\big(\Psi'(\theta)\cos\theta\big)' = -\lambda\Psi(\theta)\cos\theta, \quad \theta \in (\theta_1,\theta_2),  \\
 &\Psi(\theta_1) = \Psi(\theta_2) = 0.
 \end{split}
 \end{align}
Then the inhomogeneous problem \eqref{radial_problem} has a unique solution $\Psi_{\lambda,\Upsilon}$ for all $\lambda \notin \{\lambda_n : n\in\mathbb{N}\}$.
\end{proposition}
\begin{proof}
By setting
 \[\tilde{\Psi}(\theta) = \Psi_{\lambda,\Upsilon}(\theta) - (a\theta+b),\]
where
 \[a = \frac{\psi_2 - \psi_1}{\theta_2-\theta_1} \quad \text{and} \quad b = \frac{\theta_2\psi_1 - \theta_1\psi_2}{\theta_2-\theta_1},\]
problem \eqref{radial_problem} transforms to
 \begin{align*}
 &\big(\tilde{\Psi}'(\theta)\cos\theta\big)' = -\lambda\tilde{\Psi}(\theta)\cos\theta + a\sin\theta + [\Upsilon - \lambda(a\theta+b)]\cos\theta - \omega\sin 2\theta, \quad \theta \in (\theta_1,\theta_2),  \\
 &\tilde{\Psi}(\theta_1) = \tilde{\Psi}(\theta_2) = 0,
 \end{align*}
which has homogeneous Dirichlet boundary conditions. The claim then follows from Lemma \ref{lemma_inhomogeneous}.
\end{proof}

\begin{remark} \label{remark}
Note that, by the weak maximum principle (see, e.g., \cite{GilTru01}), solutions to problem \eqref{stationary_radial} are unique if $\lambda \leq 0$. Therefore, the radial solutions for $\lambda \leq 0$ are in fact \emph{all} solutions for such values of $\lambda$.
\end{remark}

\begin{remark}
It is possible to infer that $\lambda\notin\{\lambda_n : n\in\mathbb{N}\}$ if the annular domain $\C$ is sufficiently ``narrow''; although this is not going to be relevant for the rest of the paper, a brief discussion of this fact can nonetheless be found in Appendix \ref{Appendix}.
\end{remark}

\subsection{Stability} \label{stability}

We begin the discussion of stability by identifying certain conserved quantities.

\begin{proposition} \label{conserved_quantities}
Let $f\in C^1(\R)$ be arbitrary. Then the quantities
 \begin{subequations}
 \begin{gather} 
 \frac{1}{2}\iint_\C\bigg(\psi_\theta^2(\varphi,\theta,t) + \frac{1}{\cos^2\theta}\psi_\varphi^2(\varphi,\theta,t)\bigg)\dsi, \label{conserved_quantity_1}\\
 \int_0^{2\pi}\psi_\theta(\varphi,\theta_1,t)\dphi, \quad \int_0^{2\pi}\psi_\theta(\varphi,\theta_2,t)\dphi \label{conserved_quantity_2}
 \end{gather}
and
 \begin{equation} \label{conserved_quantity_3}
 \iint_\C f(\Delta\psi(\varphi,\theta,t) + 2\omega\sin\theta)\dsi
 \end{equation}
 \end{subequations}
are conserved by smooth solutions of \eqref{main_time_dependent} with boundary condition \eqref{BC_polar}.
\end{proposition}
\begin{remark}
The expression \eqref{conserved_quantity_1} is the kinetic energy, the quantities in \eqref{conserved_quantity_2} are the circulations along the ``lower" and ``upper" boundary, and \eqref{conserved_quantity_3} corresponds to the so-called \emph{Casimir invariants} (see \cite{ConGer22}).
\end{remark}
\begin{proof}
First of all, evaluating \eqref{Euler_spherical_1} on $\{\theta = \theta_i\}$ ($i\in\{1,2\}$), and taking \eqref{BC_spherical} into account, we see that
 \[\psi_{\theta t}|_{\theta=\theta_i} = - u_t|_{\theta=\theta_i} = \bigg(\frac{u^2}{2\cos\theta}\bigg|_{r=r_1} + \frac{p}{\cos\theta}\bigg|_{r=r_1}\bigg)_\phi,\]
hence the quantities in \eqref{conserved_quantity_2} are conserved. Consequently, using \eqref{main_time_dependent}, the boundary condition \eqref{BC_polar} (which implies that $\psi_\varphi = 0$ on the boundary), and the $2\pi$-periodicity in the $\varphi$-argument, we see via integration by parts that
 \begin{align*}
 \p_t\iint_\C\frac{1}{2}\bigg(\psi_\theta^2 + \frac{1}{\cos^2\theta}\psi_\varphi^2\bigg)\dsi &= \iint_\C\bigg(\psi_\theta\psi_{\theta t} + \frac{1}{\cos^2\theta}\psi_\varphi\psi_{\varphi t}\bigg)\dsi \\
 &= \psi_2\cos\theta_2\int_0^{2\pi}\psi_{\theta t}|_{r=r_2}\dphi - \psi_1\cos\theta_1\int_0^{2\pi}\psi_{\theta t}|_{r=r_1}\dphi \\
 &\quad - \iint_\C\psi(\Delta\psi)_t\dsi \\
 &= \int_0^{2\pi}\int_{\theta_1}^{\theta_2}\psi\big[\psi_\varphi \p_\theta - \psi_\theta \p_\varphi\big](\Delta\psi + 2\omega\sin\theta)\dtheta\dphi = 0;
 \end{align*}
this establishes the claim for \eqref{conserved_quantity_1}. As for \eqref{conserved_quantity_3}, using \eqref{main_time_dependent} and \eqref{BC_polar}, it follows
 \begin{align*}
 &\p_t\iint_\C f(\Delta\psi + 2\omega\sin\theta)\dsi \\
 &\quad\quad = \int_0^{2\pi}\int_{\theta_1}^{\theta_2}\big\{\psi_\theta\big[f(\Delta\psi + 2\omega\sin\theta)\big]_\varphi - \psi_\varphi\big[f(\Delta\psi + 2\omega\sin\theta)\big]_\theta\big\}\dtheta\dphi = 0;
 \end{align*}
this completes the proof.
\end{proof}

In view of Proposition \ref{conserved_quantities}, we see that the functional
 \begin{align*}
 \E_{\alpha,\beta}(\psi) &= \frac{1}{2}\iint_\C\bigg[-\lambda\bigg(\psi_\theta^2 + \frac{\psi_\varphi^2}{\cos^2\theta}\bigg) + (\Delta\psi + 2\omega\sin\theta - \Upsilon)^2\bigg]\dsi \\
 &\quad + \lambda\left\{\alpha\int_0^{2\pi}\psi_\theta|_{\theta=\theta_2}\dphi + \beta\int_0^{2\pi}\psi_\theta|_{\theta=\theta_1}\dphi\right\}
 \end{align*}
is constant along smooth solutions to \eqref{main_time_dependent} and \eqref{BC_polar} for any $\alpha,\beta\in\R$. For such a solution $\psi$, we may compute the first and second variation of $\E_{\alpha,\beta}(\psi)$, to find
 \begin{align*}
 \frac{{\rm d}}{{\rm d}s}\E_{\alpha,\beta}(\psi+s\xi)\bigg|_{s=0} &= \iint_\C(\lambda\psi - \Upsilon + \Delta\psi + 2\omega\sin\theta)\Delta\xi\dsi \\
 &\quad + \lambda\bigg\{(\alpha-\psi_2\cos\theta_2)\int_0^{2\pi}\xi_\theta|_{\theta=\theta_2}\dphi \\
 &\hspace{1.2cm} + (\beta+\psi_1(t)\cos\theta_1)\int_0^{2\pi}\xi_\theta|_{\theta=\theta_1}\dphi\bigg\}
 \end{align*}
and
 \begin{align*}
 \frac{{\rm d^2}}{{\rm d}s^2}\E_{\alpha,\beta}(\psi+s\xi)\bigg|_{s=0} &= \iint_\C\bigg[-\lambda\bigg(\xi_\theta^2 + \frac{\xi_\varphi^2}{\cos^2\theta}\bigg) + (\Delta\xi)^2\bigg]\dsi.
 \end{align*}
Therefore, wee see that, choosing
 \begin{equation} \label{alpha_beta}
 \alpha = \psi_2\cos\theta_2 \quad \text{and} \quad \beta = -\psi_1\cos\theta_1,
 \end{equation}
the zonal solution $\psi = \Psi_{\lambda,\Upsilon}$ to \eqref{radial_problem} is a critical point of the functional $\E_{\alpha,\beta}$, with non-negative second variation if $\lambda \leq 0$. These observations lead us naturally to the stability result stated in Theorem \ref{Main_Theorem_2}. Let us point out that, as is readily seen upon inspection of the upcoming proof, the conclusion \eqref{stability_conclusion} of Theorem \ref{Main_Theorem_2} actually holds for all values of $\lambda$ for which $\Psi_{\lambda,\Upsilon}$ is defined, although clearly this implies stability only for $\lambda \leq 0$ (as hinted at by the second variation of $\E_{\alpha,\beta}$).

\begin{proof}[Proof of Theorem \ref{Main_Theorem_2}]
The idea of this  proof is similar to \cite[pp.\ 104--109]{MarPul94}), but we present it nonetheless for the sake of completeness. We choose $\alpha$ and $\beta$ as in \eqref{alpha_beta}. For notational convenience, we will write $\E_{\alpha,\beta} = \E$ and $\psi^* = \Psi_{\lambda,\Upsilon}$ throughout the proof. Let $\psi$ be the stream function corresponding to $(u,v)$ (as in \eqref{stream_function}) solving \eqref{main_time_dependent}--\eqref{BC_polar}. It is immediate to check that
 \[\0 = \Delta\psi.\]
Using the identity
$\frac{1}{2}(a^2-b^2) = \frac{1}{2}(a-b)^2 + b(a-b)$, valid for all $a,b\in\R$, we see that
 \begin{align*}
 \E(\psi)-\E(\psi^*) &= -\frac{\lambda}{2}\iint_\C\bigg[(\psi_\theta-\psi^*_\theta)^2 + \frac{1}{\cos^2\theta}(\psi_\varphi-\psi^*_\varphi)^2\bigg]\dsi \\
 &\quad -\lambda\iint_\C\bigg[\psi^*_\theta(\psi_\theta-\psi^*_\theta) + \frac{1}{\cos^2\theta}\psi^*_\varphi(\psi_\varphi-\psi^*_\varphi)\bigg]\dsi \\
 &\quad + \frac{1}{2}\iint_\C(\Delta\psi - \Delta\psi^*)^2\dsi + \iint_\C(\Delta\psi^* + 2\omega\sin\theta - \Upsilon)(\Delta\psi-\Delta\psi^*)\dsi \\
 &\quad + \lambda\psi_2\cos\theta_2\int_0^{2\pi}(\psi_\theta-\psi_\theta^*)|_{\theta=\theta_1}\dphi -  \lambda\psi_1\cos\theta_1\int_0^{2\pi}(\psi_\theta-\psi_\theta^*)|_{\theta=\theta_2}\dphi\\
 &= -\frac{\lambda}{2}\|\mathbf{u} - \mathbf{u}^*_{\lambda,\Upsilon}\|^2_{L^2(\C)} + \frac{1}{2}\|\Delta\psi - \Delta\psi^*\|^2_{L^2(\C)},
 \end{align*}
in view of \eqref{stream_function}. Therefore, since $\E$ is a first integral for \eqref{Euler_spherical}--\eqref{BC_spherical},
 \begin{align*}
 -\lambda\|\mathbf{u}(t) - \mathbf{u}^*_{\lambda,\Upsilon}\|^2_{L^2(\C)} + \|\0(t)-\0^*_{\lambda,\Upsilon}\|^2_{L^2(\C)} &= 2(\E(\psi(t))-\E(\psi^*)) \\
 &= 2(\E(\psi_0) - \E(\psi^*)) \\
 &= -\lambda\|\mathbf{u}_0 - \mathbf{u}^*_{\lambda,\Upsilon}\|^2_{L^2(\C)} + \|\0_0-\0^*_{\lambda,\Upsilon}\|^2_{L^2(\C)},
 \end{align*}
as claimed.
\end{proof}

\begin{remark}
In particular, taking $\lambda = 0$, we have stability (albeit exclusively in terms of the vorticity) of the solution corresponding to the stream function $\Psi_{0,\Upsilon}$ which solves the problem
 \begin{subequations} \label{zonal_problem_simple}
 \begin{align}
 &\big(\Psi_{0,\Upsilon}'(\theta)\cos\theta\big)' = \Upsilon\cos\theta - \omega\sin2\theta, \quad \theta \in (\theta_1,\theta_2), \label{zonal_problem_simple_1}\\
 &\Psi_{0,\Upsilon}(\theta_1) = \psi_1, \ \Psi_{0,\Upsilon}(\theta_2) = \psi_2 \label{zonal_problem_simple_2}
 \end{align}
 \end{subequations}
(existence and uniqueness of $\Psi_{0,\Upsilon}$ are guaranteed by Proposition \ref{one-parameter_family}). It is interesting to note that, in this special case, one can find infinitely many Lyapunov functions. Indeed, for $n\in\mathbb{N}$, consider the functional
 \[\E_n(\psi) = \iint_\C \bigg[-\Upsilon\frac{n+1}{n}(\Delta\psi + 2\omega\sin\theta)^n + (\Delta\psi + 2\omega\sin\theta)^{n+1}\bigg]\dsi,\]
which is conserved along solutions of \eqref{main_time_dependent}--\eqref{BC_polar} by Proposition \ref{conserved_quantities}. A simple calculation yields
 \[\frac{{\rm d}}{{\rm d}s}\E_n(\Psi_{0,\Upsilon}+s\xi)\bigg|_{s=0} = 0 \quad \text{and} \quad \frac{{\rm d}^2}{{\rm d}s^2}\E_n(\Psi_{0,\Upsilon}+s\xi)\bigg|_{s=0} = (n+1)\Upsilon^{n-1}\iint_\C(\Delta\xi)^2\dsi,\]
thus $\Psi_{0,\Upsilon}$ is a critical point of $\E_n$ for each $n\in\mathbb{N}$ and the second variation is non-negative provided that $n$ be even and $\Upsilon\geq 0$, or $n$ odd and $\Upsilon$ of any sign. The solution of \eqref{zonal_problem_simple} can in fact be computed explicitly by integrating \eqref{zonal_problem_simple_1} twice; this gives
 \[\Psi_{0,\Upsilon}(\theta) = \zeta(\theta) + c_1\eta(\theta) + c_2,\]
where
 \begin{align*}
 & \zeta(\theta) = -\Upsilon\log(\cos\theta)+\omega\sin\theta-\frac{\omega}{2}\tanh^{-1}(\sin\theta), \\
 &\eta(\theta) = \log\bigg(\sin\frac{\theta}{2} + \cos\frac{\theta}{2}\bigg)-\log\bigg(\sin\frac{\theta}{2} - \cos\frac{\theta}{2}\bigg)
 \end{align*}
and the integration constants $c_1,c_2$ are determined through the boundary values \eqref{zonal_problem_simple_2} as
 \begin{align*}
 c_1 &= \frac{\psi_2 - \psi_1 + \zeta(\theta_1) - \zeta(\theta_2)}{\eta(\theta_2) - \eta(\theta_1)}, \\
 c_2 &= \frac{\psi_1\eta(\theta_2) - 2\psi_1\eta(\theta_1) + \psi_2\eta(\theta_1) - \zeta(\theta_1)\eta(\theta_2) + 2\eta(\theta_1)\zeta(\theta_1) - \zeta(\theta_2)\eta(\theta_1)}{\eta(\theta_2) - \eta(\theta_1)}.
 \end{align*}
\end{remark}

\section{Discussion and final remarks}

\begin{figure}[!b]
    \centering
    \includegraphics[width=0.5\linewidth]{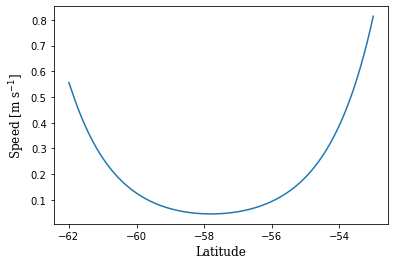}
    \caption{The (dimensional) zonal velocity corresponding to the (nondimensional) stream function $\Psi_{\lambda,\Upsilon}$ for the values $\lambda = -3000$, $\Upsilon = 30000$, $\psi_1 = -5$, and $\psi_2 = -25$.}
    \label{fig:velocity}
\end{figure}

The objective of this paper has been to discuss some qualitative features of a model for the Antarctic Circumpolar Current (ACC), including most notably stability properties of a class of zonal solutions. This latter fact may contribute to explaining why the current pattern observed in the ACC is persistent. To this end, it is interesting to inquire whether the large family of solutions $\Psi_{\lambda,\Upsilon}$ of \eqref{radial_problem} does indeed contain solutions (for some $\lambda \leq 0$) which exhibit at least some features that emerge from on-site measurement of the ACC. As mentioned in the Introduction, the ACC is in fact composed by fast and vertically coherent jets, with speeds which can exceed 1$\,{\rm m \, s}^{-1}$, separated by zones of low-speed flow. In particular, the ACC is delimited by two such jets, namely the inner Polar Front and the outer Subantarctic Front, which are situated at approximate latitudes of $60^\circ$S and $50^\circ$S. The presence of such jets at the boundary of the region, separated by a low-speed region in the middle, is indeed captured by $\Psi_{\lambda,\Upsilon}$ for appropriate choices of $\lambda$, $\psi_1$ and $\psi_2$; see Figure \ref{fig:velocity}.

\appendix

\section{Proof of existence and uniqueness of classical solutions} \label{appendix_classical_solutions}

In proving the global well-posedness result stated in Theorem \ref{global_well-posedness}, we will follow closely the classical paper \cite{Kat67}; since numerous proofs---especially of technical lemmas---carry over almost (if not entirely) word-for-word to our situation, we shall often omit them and instead simply refer the reader to that work. Nevertheless, some care is needed at  several places, where the difference between our equations \eqref{Euler_compact} and the ``standard" Euler equations in the plane (which would correspond to having $\alpha = 1$ and $\beta = 0$ in \eqref{Euler_compact_1}) will play a role; this will be the main focus of the subsequent discussion. The domain $D$, the initial condition $\U_0$, and $T>0$ are supposed to be fixed.

Some more notation will be needed throughout this section. If $j,k\in\mathbb{N}\cup\{0\}$ and $\varepsilon,\delta\in [0,1)$, the class $C^{j,k}(\overline{Q}_T)$ is the set of all (scalar- or vector-valued) functions whose derivatives of order $p$ in $\x$ and of order $q$ in $t$ exists and are continuous for all integers $0\leq p \leq j$ and $0\leq q \leq k$, whereas $C^{j+\varepsilon,k+\delta}(\overline{Q}_T)$ is the set of all (scalar- or vector-valued) functions whose derivatives of order $p$ in $\x$ and of order $q$ in $t$ exists and are H\"older continuous with exponent $\varepsilon$ in $\x$ (if $\varepsilon>0$) or with exponent $\delta$ in $t$ (if $\delta>0$), or both (if $\varepsilon>0$ and $\delta>0$) for all integers $0\leq p \leq j$ and $0\leq q \leq k$.

The Euclidean scalar product of two vectors ${\bf f},{\bf g}$ is denoted by ${\bf f}\cdot{\bf g}$; we denote by $\|{\bf f}\|_2 = \sqrt{{\bf f}\cdot{\bf f}}$ the Euclidean norm, as before. Moreover, we will write
 \[(\varphi,\psi) = \iint_D\varphi(x,y)\psi(x,y)\dx\dy \quad \text{and} \quad ({\bf f},{\bf g}) = \iint_D{\bf f}(x,y)\cdot{\bf g}(x,y)\dx\dy\]
for scalar functions $\varphi,\psi$ and vector functions ${\bf f},{\bf g}$.

A vector-valued function ${\bf f}$ with $\nabla\cdot{\bf f}=0$ will be referred to as a \emph{flow}. If, in addition, ${\bf f}\cdot\n = 0$ on $\p D$ (where $\n$ is the outward normal derivative on $\p D$), we will call ${\bf f}$ a \emph{tangential flow}.

Following the notation in \cite{Kat67}, we will denote by $L_1(z),L_2(z),\ldots$ monotone increasing functions of $z\geq 0$, which may depend only on $D$, $\U_0$, and $T$.

\subsection{Technical preparation}

We begin by stating and proving a lemma (which is a slight generalisation of Lemma 1.1 in \cite{Kat67}) that will be useful later on.

\begin{lemma} \label{lemma_identity}
Let $\U\in C^1(\overline{D})$ be a flow and set $\varphi = \nabla\times\U$; also, let $f\in C^1(\overline{D})$. Then
 \[(f\varphi,\U\cdot\nabla\Phi) + ((f\U\cdot\nabla)\U,\nabla^\perp\Phi) + \left(\frac{|\U|^2}{2}\nabla f,\nabla^\perp\Phi\right) = 0\]
for all $\Phi\in C^1(\overline{D})$ with $\Phi|_{\p D} = 0$.
\end{lemma}
\begin{proof}
If $\U\in C^2(\overline{D})$, the claim is a consequence of the easily verified identity
 \begin{align*}
 f\varphi\U\cdot\nabla\Phi + ((f\U\cdot\nabla)\U)\cdot\nabla^\perp\Phi &= \nabla\cdot(f\varphi\Phi\U) - \nabla\times((\Phi f\U\cdot\nabla)\U) \\
 &\quad - \varphi\Phi\U\cdot\nabla f - ((\Phi\U\cdot\nabla)\U)\cdot\nabla^\perp f,
 \end{align*}
together with the observation that
 \begin{align*}
 (\varphi\Phi\U,\nabla f) + ((\Phi\U\cdot\nabla)\U),\nabla^\perp f) &= \left(\nabla\bigg(\Phi\frac{|\U^2|}{2}\bigg) - \frac{|\U|^2}{2}\nabla\Phi,\nabla^\perp f\right) \\
 &= \left(\frac{|\U|^2}{2}\nabla f,\nabla^\perp\Phi\right).
 \end{align*}
The result for $\U\in C^1(\overline{D})$ follows from here by approximating $\U$ by means of the usual mollifiers (analogously to \cite[Lemma 1.1]{Kat67}).
\end{proof}

We may eliminate the pressure $P$ from \eqref{Euler_compact_1}--\eqref{Euler_compact_2}, to formally obtain the vorticity equation
 \begin{equation}\label{vorticity_equation}
 \xi_t + \U\cdot\nabla(\alpha\xi + \beta) = \xi_t + \nabla\cdot[(\alpha\xi + \beta)\U] = 0,
 \end{equation}
where
 \[\xi = \nabla\times\U.\]
It is in fact convenient to multiply \eqref{vorticity_equation} by $\alpha$ and use the time-independence of $\alpha$ and $\beta$, to rewrite \eqref{vorticity_equation} as
 \begin{equation} \label{vorticity_equation_2}
 \zeta_t + (\alpha\U)\cdot\nabla\zeta = 0,
 \end{equation}
where
 \[\zeta \coloneqq \alpha\xi + \beta.\]

We will denote by $g : \R^2\times\R^2 \supseteq D\times D\to\R$ the Green function of the domain $D$. The integral operator with kernel $g$ is denoted by $G$:
 \[G\varphi(\x) = \iint_D g(\x,\tilde{\x})\varphi(\tilde{\x})\,{\rm d}\tilde{\x}, \quad {\bf x} \in D.\]
Then, as is known, $\psi = G\varphi$ is the solution of the Poisson boundary value problem $\Delta\psi = -\varphi$, $\psi|_{\p D} = 0$. Let us recall some useful estimates and regularity properties for $g$ and $G$.

\begin{lemma} \label{lemma_regularity}
 \begin{itemize}
  \item[(i)] Let $\varepsilon,\delta\in (0,1)$. The map $C^{0+\delta}(\overline{D})\to C^{1+\delta}(\overline{D})$, $\varphi\mapsto \nabla^\perp G\varphi$ is continuous. If $\varphi\in C^{0+\delta,0}(\overline{Q}_T)$, then $\nabla^\perp G\varphi\in C^{1+\delta',0}(\overline{Q}_T)$ for all $\delta'\in (0,\delta)$. If $\varphi\in C^{0+\delta,0+\varepsilon}(\overline{Q}_T)$, then $\nabla^\perp G\varphi\in C^{1+\delta',\varepsilon'}(\overline{Q}_T)$ for all $\delta'\in (0,\delta)$ and $\varepsilon'\in (0,\varepsilon)$.
  \item[(ii)] If $\varphi\in L^\infty(D)$, then $G\varphi\in C^1(\overline{D})$ and $\U = \nabla^\perp G\varphi$ is a tangential flow with
   \[\|\U\|_\infty \leq C\|\varphi\|_\infty,\]
  where $C>0$ depends only on $D$.
 \end{itemize}
\end{lemma}

The first step will be to show that, for all scalar functions $\varphi = \varphi(\x,t)$ in a certain class $S'$, we can construct a unique globally-in-time tangential flow $\U = U\ee_x + V\ee_y$ such that
 \begin{align}
 & \nabla\times\U = \varphi, \label{flow_1} \\
 & \p_t(\U,\U^*) + \left((\alpha\U\cdot\nabla)\U + \frac{|\U|^2}{2}\nabla\alpha + \beta J\U,\U^*\right) = 0, \label{flow_2} \\
 & (\U|_{t=0} - \U_0, \U^*) = 0, \label{flow_3}
 \end{align}
where $\U^*\in C^1(\overline{D})$ is defined as
 \begin{equation} \label{U*}
 \U^* = \frac{\nabla^\perp\psi^*}{(\nabla^\perp\psi^*,\nabla^\perp\psi^*)},
 \end{equation}
$\psi^*\in C^2(\overline{D})$ being the unique harmonic function with $\psi^*|_{\Gamma_1} = 1$ and $\psi^*|_{\Gamma_2} = 0$ (with $\Gamma_1 = \{(x,y) : x^2+y^2 = r_1^2\}$ and $\Gamma_2 = \{(x,y) : x^2+y^2 = r_2^2\}$).

The next technical auxiliary result, which is going to be needed in the sequel, is \cite[Lemma 1.6]{Kat67}.

\begin{lemma} \label{lemma_pressure}
Let ${\bf f}\in C(\overline{D})$ be a vector-valued function such that
 \[({\bf f},\U^*) = 0 \quad \text{and} \quad ({\bf f},\nabla^\perp\Phi) = 0\]
for all $\Phi\in C^\infty_{\rm c}(D)$ (where $\U^*$ is as in \eqref{U*}). Then there exists a function $p\in C^1(\overline{D})$ such that ${\bf f} = \nabla p$. If ${\bf f}\in C(\overline{Q}_T)$ and the above conditions are satisfied for each $t\in[0,T]$, then $p$ can be chosen from $C^{1,0}(\overline{Q}_T)$.
\end{lemma}

We now define the class
 \[S' = \bigcup_{\varepsilon\in (0,1)}C^{\varepsilon,0}(\overline{Q}_T).\]

\begin{lemma} \label{lemma_flow_construction}
For all $\varphi\in S'$ there exists a unique $\U\in C^{1,0}(\overline{Q}_T)$ satisfying \eqref{flow_1}--\eqref{flow_3}; $\U$ is a tangential flow and can be written as $\U = \overline{\U} + \U'$, with
 \begin{equation} \label{U}
 \overline{\U} = \nabla^\perp G\varphi \quad \text{and} \quad \U' = \lambda(t)\U^*,
 \end{equation}
where $|\lambda(t)| \leq L_1(\|\varphi\|_\infty)$.
\end{lemma}
\begin{proof}
Since $\varphi\in C^{\varepsilon,0}(\overline{D})$ for some $\varepsilon>0$, by Lemma \ref{lemma_regularity} we have $G\varphi\in C^{2,0}(\overline{Q}_T)$ and $\overline{\U} = \nabla^\perp G\varphi \in C^{1,0}(\overline{Q}_T)$. Clearly, $\nabla\times\overline{\U} = -\Delta G\varphi = \varphi$ and $\nabla\times\U^* = 0$, thus \eqref{flow_1} is satisfied by $\U$ as in \eqref{U}, for any choice of $\lambda(t)$. We now determine $\lambda(t)$ so that \eqref{flow_2} and \eqref{flow_3} are satisfied as well. Note that $(\overline{\U},\U^*) = 0$, as is seen by integrating by parts and by $\nabla\times\U^* = 0$ and $G\varphi|_{\p D} = 0$. Plugging $\U = \overline{\U} + \U'$ into \eqref{flow_2} and \eqref{flow_3} yields for $\lambda(t)$ the ordinary differential equation
 \[\lambda'(t) + \mu_1(t)\lambda(t)^2 + \mu_2\lambda(t)^2 + \mu_3(t) = 0,\]
with initial condition
 \[\lambda(0) = (\U_0,\U^*);\]
here, $\mu_2$ is a real constant, whereas $\mu_1(t)$ and $\mu_3(t)$ are continuous functions of $t$. Standard ODE theory yields the existence of a \emph{local} solution $\lambda(t)$. Next we show that this solution can be extended to the whole interval $[0,T]$. To this end, let $[0,\tau)$ denote the maximal interval of existence, for an appropriate $\tau>0$. We multiply \eqref{flow_2} by $\lambda(t),$ so that
 \[(\p_t\U,\U') + \left((\alpha\U\cdot\nabla)\U + \frac{|\U|^2}{2}\nabla\alpha + \beta J\U,\U'\right) = 0.\]
Plugging in $\U = \overline{\U} + \U'$, and using again $(\overline{\U},\U^*) = 0$ and the identities
 \begin{align*}
 ((\alpha\U\cdot\nabla)\U',\U') &= -\frac{1}{2}(\U\cdot\nabla\alpha,|\U'|^2), \\
 ((\alpha\U\cdot\nabla)\overline{\U},\U') &= -((\alpha\U\cdot\nabla)\U',\overline{\U}) - (\U\cdot\nabla\alpha,\U'\cdot\overline{\U}),
 \end{align*}
we obtain the equation
 \[\frac{1}{2}\frac{{\rm d}}{{\rm d}t}[\lambda(t)^2] + \gamma_1(t)\lambda(t) + \gamma_2(t)\lambda(t)^2 = 0,\]
where the coefficients\footnote{The terms which would make up the coefficient of $\lambda(t)^3$ cancel exactly.}
 \[\gamma_1(t) = -((\alpha\overline{\U}\cdot\nabla)\U^*,\overline{\U}) - (\overline{\U}\cdot\nabla\alpha,\U^*\cdot\overline{\U}) + \frac{1}{2}(|\overline{\U}|^2\nabla\alpha,\U^*) + (\beta J\overline{\U},\U^*)\]
and
 \begin{align*}
 \gamma_2(t) &= -((\alpha\U^*\cdot\nabla)\U^*,\overline{\U}) - (\U^*\cdot\nabla\alpha,\U^*\cdot\overline{\U}) - \frac{1}{2}(\overline{\U}\cdot\nabla\alpha,|\U^*|^2) \\
 &\quad + ((\overline{\U}\cdot\U^*)\nabla\alpha,\U^*) + (\beta J\U^*,\U^*)
 \end{align*}
depend on $\overline{\U}$, but not on its derivatives, and---in view of Lemma \ref{lemma_regularity}---are therefore readily estimated in terms of the norm $\|\varphi\|_\infty$:
 \[|\gamma_1(t)| \leq L_2(\|\varphi\|_\infty), \quad |\gamma_2(t)| \leq L_3(\|\varphi\|_\infty).\]
Hence, with $C(\|\varphi\|_\infty) = \max\{L_2(\|\varphi\|_\infty),L_3(\|\varphi\|_\infty)\}$,
 \begin{align*}
 \lambda(t)^2 &= \lambda(0)^2 - 2\int_0^t\big[\gamma_1(s)\lambda(s) + \gamma_2(s)\lambda(s)^2\big]\ds \\
 &\leq \lambda(0)^2 + 2C(\|\varphi\|_\infty)\int_0^t\big[|\lambda(s)| + |\lambda(s)|^2\big]\ds \\
 &\leq \lambda(0)^2 + C(\|\varphi\|_\infty)t + 3C(\|\varphi\|_\infty)\int_0^t\lambda(s)^2\ds,
 \end{align*}
and Gronwall's inequality now implies
 \[\lambda(t)^2 \leq (\lambda(0)^2 + C(\|\varphi\|_\infty)\tau)\e^{3C(\|\varphi\|_\infty)\tau} \quad \text{for all } t\in [0,\tau).\]
Therefore $\lambda(t)$ is uniformly bounded on the interval $[0,\tau)$, for any $\tau\in[0,T]$, and so can be extended to the whole of $[0,T]$. Setting
 \[L_1(\|\varphi\|_\infty) = [(\U_0,\U^*)^2 + C(\|\varphi\|_\infty)T]^{1/2}\e^{3C(\|\varphi\|_\infty)T/2}\]
finishes the proof.
\end{proof}

Next, given $\varphi\in S'$ and $\U$ defined according to Lemma \ref{lemma_flow_construction}, we turn to solving the initial value problem for the transport equation
 \begin{align} \label{transport_equation}
 \begin{split}
 & \zeta_t + (\alpha\U)\cdot\nabla\zeta = 0, \\
 & \zeta|_{t=0} = \zeta_0,
 \end{split}
 \end{align}
where $\zeta_0 = \alpha\nabla\times\U_0 + \beta$. This is solved via a straightforward application of the method of characteristics. Given $t\in[0,T]$ and $\x = (x,y) \in D$, we will denote by
\[s\mapsto {\bf X}^t(\x;s) = (X^t(x,y;s),Y^t(x,y;s))\]
the characteristic line for which ${\bf X}^t(\x;t) = \x$. In other words, ${\bf X}^t(\x;\cdot)$ solves the ``initial" value problem
 \begin{align} \label{characteristics}
 \begin{split}
 & \frac{{\rm d}}{{\rm d}s}{\bf X}^t(\x;s) = \alpha({\bf X}^t(\x;s))\U({\bf X}^t(\x;s),s), \\
 & {\bf X}^t(\x;t) = \x.
 \end{split}
 \end{align}
Since $\alpha\U\in C^{1,0}(Q_T)$, existence of a local unique solution is guaranteed by standard theory. Due to the fact that $\alpha\U$ has zero normal component on $\p D$, the solution can be extended to the whole interval $s\in[0,T]$, and has useful regularity properties:

\begin{lemma}
Given $t\in[0,T]$ and $\x\in D$, the problem \eqref{characteristics} has a unique solution ${\bf X}^t(\x;s)$ defined for all $s\in[0,T]$; this solution is continuously differentiable in all three arguments. For fixed $t$ and $s$, ${\bf X}^t(\cdot;s)$ is a one-to-one, measure preserving map of $\overline{D}$ onto itself, with its Jacobian determinant equal to $1$, where $\p D$ is mapped onto itself. ${\bf X}^s(\cdot;s)$ is the identity map, and ${\bf X}^s(\cdot;t)$ is the inverse map of ${\bf X}^t(\cdot;s)$.
\end{lemma}
\begin{proof}
Due to the fact that $\alpha\U \in C^{1,0}(Q_T)$ is tangential, the proof is word-for-word identical to those of \cite[Lemma 2.2 \& Lemma 2.3]{Kat67}.
\end{proof}

Given the characteristic lines ${\bf X}^t(\x;s)$, the solution of \eqref{transport_equation} at $(x,y,t) \in Q_T$ is then simply given by
 \begin{equation} \label{zeta}
 \zeta(x,y,t) = \zeta_0(X^t(x,y;0),Y^t(x,y;0)),
 \end{equation}
or, in terms of $\xi = \frac{1}{\alpha}\zeta - \frac{\beta}{\alpha}$,
 \[\xi(x,y,t) = \frac{\zeta_0(X^t(x,y;0),Y^t(x,y;0))}{\alpha(x,y)} - \frac{\beta(x,y)}{\alpha(x,y)}.\]
Since ${\bf X}^t(\x;s)$ is determined by $\U$, which is determined by $\varphi$, this means that $\xi$ is determined by $\varphi$ as well, so we may write
 \[\xi = F[\varphi],\]
$F$ being a map from $S'$ to $C(\overline{Q}_T)$. It is important to remark that, as is known, $\zeta$ given by \eqref{zeta} is a classical solution of \eqref{transport_equation} only if $\zeta_0\in C^1(\overline{D})$; however, if this is not the case, then $\zeta$ is still a weak solution, in the following sense (cf. \cite[Lemma 2.4]{Kat67}):

\begin{lemma} \label{lemma_vorticity}
For all $\Phi\in C^1(\overline{D})$, we have
 \[\p_t(\zeta,\Phi) = (\zeta,\alpha\U\cdot\nabla\Phi).\]
\end{lemma}

Let us now list some more estimates that will be needed later on. The proof is essentially identical to the proofs of Lemmas 2.5, 2.6, and 2.7 in \cite{Kat67} (up to occasional trivial adjustments to account for $\alpha$ and $\beta$) and is therefore omitted.

\begin{lemma} \label{lemma_estimates}
Given $\varphi\in S'$, $\xi = F[\varphi]$ satisfies the inequalities
  \[ |\xi(\x,t)| \leq A\|\zeta_0\|_\infty + B,\]
where
 \[A = \sup_{\x\in D}\frac{1}{\alpha(\x)} = \frac{4}{(1+r_1^2)^2} \quad \text{and} \quad B = \sup_{\x\in D}\frac{\beta(\x)}{\alpha(\x)} = 8\omega\frac{1-r_1^2}{(1+r_1^2)^3},\]
and
 \[|\xi(\x_1,t) - \xi(\x_2,t)| \leq L_4(\|\varphi\|_\infty)(\|\x_1 - \x_2\|_2^\delta + |t - s|^\delta),\]
where $\delta^{-1} = L_5(\|\varphi\|_\infty)$, whenever $\|\x_1 - \x_2\|_2 \leq 1$ and $|t - s| \leq 1$.
\end{lemma}

We are now in a position to define a new class $S$ that will allow us to apply Schauder's fixed-point theorem.

\begin{definition}
For $M = A\|\zeta_0\|_\infty + B$, we define the class $S$ as the set of all functions $\varphi : D\times [0,T]\to\R$ such that $\|\varphi\|_\infty \leq M$ and
 \[|\varphi(\x_1,t) - \varphi(\x_2,t)| \leq L(\|\x_1 - \x_2\|_2^\delta + |t - s|^\delta) \quad \text{for $\|\x_1 - \x_2\|_2 \leq 1$, $|t - s| \leq 1$},\]
where $L = L_4(\|\varphi\|_\infty)$ and $\delta^{-1} = L_5(\|\varphi\|_\infty)$ (here it is crucial that $L$ and $\delta$ depend on $\|\varphi\|_\infty$ but not on the H\"older exponents of $\varphi$).
\end{definition}

Clearly, $S\subseteq S'$, and, by Lemma \ref{lemma_estimates}, $F$ maps $S$ into itself. It is also clear that $S$ is a compact convex subset of the Banach space $C(\overline{Q}_T)$. Moreover, we have (entirely analogously to \cite[Lemma 2.8]{Kat67}):

\begin{lemma}
The map $F$ is continuous on $S$ in the topology of $C(\overline{Q}_T)$.
\end{lemma}

Schauder's fixed-point theorem now yields the existence of a fixed point $\varphi^*\in S$ of $F$:
 \[\xi = F[\varphi^*] = \varphi^*.\]

\subsection{Proof of Theorem \ref{global_well-posedness}}

With all this preparation, we are ready to prove Theorem \ref{global_well-posedness}. The regularity of the derivatives of $\U$ is established exactly as in \cite[Lemma 3.2 \& Lemma 3.3]{Kat67}. Let us write
 \[{\bf W} = \U_t + (\alpha\U\cdot\nabla)\U + \frac{|\U|^2}{2}\nabla\alpha + \beta J\U\]
as a short-hand notation. Let $\Phi\in C^1(\overline{D})$ with $\Phi|_{\p D} = 0$. Then, using Lemmas \ref{lemma_identity} and \ref{lemma_vorticity}, \eqref{vorticity_equation}, and the fact that $\xi = \varphi^*$, we have
 \begin{align*}
 \p_t(\U,\nabla^\perp\Phi) &= \p_t(\nabla\times\U,\Phi) = \p_t(\varphi^*,\Phi) = (\alpha\varphi^*+\beta,\U\cdot\nabla\Phi) \\
 &\quad = -((\alpha\U\cdot\nabla)\U,\nabla^\perp\Phi) - \left(\frac{|\U|^2}{2}\nabla\alpha,\nabla^\perp\Phi\right) - (\beta J\U,\nabla^\perp\Phi),
 \end{align*}
that is, $({\bf W},\nabla^\perp\Phi) = 0$ for each such $\Phi$. On the other hand, $({\bf W},\U^*) = 0$ by \eqref{flow_2}. Lemma \ref{lemma_pressure} now yields the existence of $P\in C^{1,0}(\overline{Q}_T)$ such that $W = -\nabla P$; this is \eqref{Euler_compact_1}. The equations \eqref{Euler_compact_2} and \eqref{Euler_compact_3} are automatically satisfied by construction of $\U$ (cf. Lemma \ref{lemma_flow_construction}). It remains to show \eqref{initial_condition}. Set ${\bf W}_0 \coloneqq \U|_{t=0} - \U_0$; then
 \[\nabla\times{\bf W}_0 = \nabla\times\U|_{t=0} - \nabla\times\U_0 = \varphi^*|_{t=0} - \nabla\times\U_0 = 0,\]
because $\varphi^* = \xi$ and $\xi|_{t=0} = \nabla\times\U_0$. Thus
 \[({\bf W}_0,\nabla^\perp\Phi) = (\nabla\times\U_0,\Phi) = 0\]
for all $\Phi\in C^1(\overline{D})$ with $\Phi|_{t=0}=0$. Moreover, $({\bf W},\U^*) = 0$ by \eqref{flow_3}. Therefore, by Lemma \ref{lemma_pressure} (and since ${\bf W}_0\in C^1(\overline{D})$) there is $q\in C^2(\overline{D})$ such that ${\bf W}_0 = \nabla q$. But ${\bf W}_0$ is a tangential flow, $q$ is harmonic in $D$ and $\nabla q\cdot{\bf n} = 0$, thus $q$ is constant and ${\bf W}_0 = 0$, as desired.

As for uniqueness, suppose that $(\U,P)$ and $(\tilde{\U},\tilde{P})$ solve \eqref{Euler_compact}--\eqref{initial_condition}. Set $\hat{\U} \coloneqq \tilde{\U} - \U$ and $Q \coloneqq \tilde{P} - P$; then $\hat{\U}\cdot{\bf n} = 0$ on $\p D$ and $\hat{\U}|_{t=0} = 0$. Subtracting the equations for $(\U,P)$ and $(\tilde{\U},\tilde{P})$, we obtain
 \[\hat{\U}_t + (\alpha\tilde{\U}\cdot\nabla)\hat{\U} + (\alpha\hat{\U}\cdot\nabla)\U + \frac{1}{2}(|\tilde{\U}|^2 - |\U|^2)\nabla\alpha + \beta J\hat{\U} = -\nabla Q.\]
Multiplying by $\hat{U}$, integrating over $D$, and noticing that
 \[((\alpha\tilde{\U}\cdot\nabla)\hat{\U},\hat{\U}) = -\frac{1}{2}(\tilde{\U}\cdot\nabla\alpha,|\hat{\U}|^2)\]
(as in the proof of Lemma \ref{lemma_flow_construction}) and $(\nabla Q,\hat{\U}) = 0$ (because $\hat{\U}$ is a tangential flow), we see that
 \[\p_t(\hat{\U},\hat{\U}) = \frac{1}{2}(\tilde{\U}\cdot\nabla\alpha,|\hat{\U}|^2) -((\alpha\hat{\U}\cdot\nabla)\U,\hat{\U}) - \frac{1}{2}(((\tilde{\U} + \U)\cdot\hat{\U})\nabla\alpha,\hat{\U}).\]
Using the fact that $\alpha,\U,\tilde{\U}$ are continuously differentiable, we can find a real (not necessarily positive) constant $C\in\R$ such that
 \[\p_t(\hat{\U},\hat{\U}) \leq C (\hat{\U},\hat{\U}).\]
By Gronwall's inequality and the initial condition $\hat{\U}|_{t=0} = 0$, it follows that $\hat{\U} = 0$ and thus $\nabla Q = 0$.

\section{An alternative approach for zonal solutions} \label{Appendix}

It turns out that, via a fixed-point argument, we can show existence of solutions to \eqref{radial_problem} for all $\lambda\in\R$, that is, it can be guaranteed that $\lambda$ is not an eigenvalue of \eqref{Sturm-Liouville_homogeneous}, provided that the region $\C$ be sufficiently ``narrow'' (depending on $|\lambda|$). The result is formulated more conveniently after applying a stereographic projection centred at the North Pole; in other words, as in Section \ref{classical_solutions}, we perform the change of variable
 \[r = \frac{\cos\theta}{1-\sin\theta},\]
after which, setting $\tilde{\Psi}_\lambda(r) = \Psi_{\lambda,\Upsilon}(\theta)$, problem \eqref{radial_problem} becomes
 \begin{subequations}\label{radial_problem_b}
 \begin{align}
 &\frac{1}{r}\big(r\tilde{\Psi}_\lambda'(r)\big)' = -\frac{4\lambda}{(1+r^2)^2}\tilde{\Psi}_\lambda(r) + 8\omega\frac{1-r^2}{(1+r^2)^3}, \quad r\in(r_1,r_2), \label{radial_equation_b} \\
 &\tilde{\Psi}_\lambda(r_1) = \psi_1, \ \tilde{\Psi}_\lambda(r_2) = \psi_2, \label{radial_BC_b}
 \end{align}
 \end{subequations}
where $r_1 = \cos\theta_1/(1-\sin\theta_1)$ and $r_2 = \cos\theta_2/(1-\sin\theta_2)$ (so that $0<r_1<r_2<1$).

\begin{proposition}
Let $\lambda\in\R$ and suppose that
 \begin{equation} \label{assumption_contraction}
 \frac{r_2}{r_1} \leq \e^{(2|\lambda|)^{-1/2}}.
 \end{equation}
Then there exists a unique solution $\tilde{\Psi}_\lambda\in C^2([r_1,r_2])$ of problem \eqref{radial_problem_b}.
\end{proposition}
\begin{proof}
Following \cite{Chu18}, it is convenient to perform the change of variable
 \[r = \e^{-t}, \quad u(t) = \tilde{\Psi}_\lambda(r) = \tilde{\Psi}_\lambda(\e^{-t}).\]
Setting
 \[t_1 = \log\frac{1}{r_2}, \quad t_2 = \log\frac{1}{r_1}\]
(so that $t_1<t_2$), equation \eqref{radial_equation_b} becomes
 \begin{equation} \label{radial_transformed}
 u''(t) = -\frac{\lambda u(t)}{\cosh^2(t)} + 2\omega\frac{\sinh(t)}{\cosh^3(t)}, \quad t\in (t_1,t_2),
 \end{equation}
whereas the boundary conditions \eqref{radial_BC_b} become simply
 \begin{equation} \label{radial_BC_transformed}
 u(t_1) = \psi_2, \quad u(t_2) = \psi_1.
 \end{equation}
Integrating \eqref{radial_transformed} twice on $[t_1,t]$ (for $t\in[t_1,t_2]$) yields the integral equation
 \begin{equation} \label{integral_equation}
 u(t) = \psi_2 - \psi_1 + \mu(u)(t-t_1) - \lambda\int_{t_1}^t\frac{t-s}{\cosh^2(s)}u(s)\ds + 2\omega\int_{t_1}^t (t-s)\frac{\sinh(s)}{\cosh^3(s)}\ds,
 \end{equation}
$\mu(u)\in\R$ being an integration constant for which we must require
 \[\mu(u) = \frac{1}{t_2-t_1}\bigg\{\psi_1-\psi_2 + \lambda\int_{t_1}^{t_2}\frac{t_2-s}{\cosh^2(s)}u(s)\ds - 2\omega\int_{t_1}^{t_2} (t_2-s)\frac{\sinh(s)}{\cosh^3(s)}\ds\bigg\}\]
for the boundary conditions \eqref{radial_BC_transformed} to be satisfied. Let us now define
 \[X = \{f\in C([t_1,t_2]) : f(t_1) = \psi_2, \, f(t_2) = \psi_1\},\]
which is a closed subset of the space $C([t_1,t_2])$ of continuous functions on $[t_1,t_2]$. We define the map $\T:X\to X$ as
 \[\T(u)(t) = \psi_2 - \psi_1 + \mu(u)(t-t_1) - \lambda\int_{t_1}^t\frac{t-s}{\cosh^2(s)}u(s)\ds + 2\omega\int_{t_1}^t (t-s)\frac{\sinh(s)}{\cosh^3(s)}\ds;\]
it is readily checked that indeed $\T(X)\subseteq X$, so that $\T$ is well-defined. We now show that $\T$ is a contraction on $X$. Let $u,v\in X$; then
 \begin{align*}
 |\T(u)(t) - \T(v)(t)| &\leq 2|\lambda|(t_2-t_1)\int_{t_1}^{t_2}|u(s)-v(s)|\ds \\
 & \leq 2|\lambda|(t_2 - t_1)^2\|u-v\|_{L^\infty([t_1,t_2])} \quad \text{for all } t\in[t_1,t_2],
 \end{align*}
which, in view of \eqref{assumption_contraction}, yields the asserted contraction property of $\T$ on $X$. The Banach fixed-point theorem now yields the existence of a unique fixed point $u\in X$ of $\T$; the fixed-point $u$ is thus a solution of the integral equation \eqref{integral_equation} satisfying the boundary conditions \eqref{radial_BC_transformed}. Finally, note that, $u$ being continuous, the right-hand side of \eqref{integral_equation} is twice continuously differentiable; this implies that, in fact, $u\in C^2([t_1,t_2])$. Transforming back to the original variable $r = \e^{-t}$ then completes the proof.
\end{proof}

\paragraph{\textbf{Data availability statement}} This research contains no data sets.

\vspace{0.2cm}

\paragraph{\textbf{Conflict of interest}} The authors have no competing financial or non-financial interests to report.

\bibliographystyle{acm}
\bibliography{references}

\end{document}